\RequirePackage{fix-cm}
\documentclass[smallextended]{svjour3}       
\smartqed  
\usepackage{graphicx}
\usepackage{lineno,hyperref,amsmath,amssymb}
\newtheorem{thm}{Theorem}

\newtheorem{lem}[thm]{Lemma}

\begin{document}

\title{Permanency and bifurcations of bounded solutions near homoclinics with symmetric eigenvalues 
}

\titlerunning{Permanency and bifurcations of bounded solutions near homoclinics}

\author{L. Soleimani         \and    O. RabieiMotlagh \and  H.M. Mohammadinejad
}



\institute{L. Soleimani  \at
             Dept. of Applied Mathematics, University of Birjand, Birjand, Iran.
              \email{l.soleimani@birjand.ac.ir}         
           \and
          O. RabieiMotlagh (corresponding author)\at
               Dept. of Applied Mathematics, University of Birjand, Birjand, Iran.
\email{orabieimotlagh@birjand.ac.ir}
\and
H.M. Mohammadinejad \at
 Dept. of Applied Mathematics, University of Birjand, Birjand, Iran.
\email{hmohammadi@birjand.ac.ir}
}

\date{Received: date / Accepted: date}

\maketitle

\begin{abstract}
We consider a system with a homoclinic orbit. We decompose the corresponding variational equation on the space of solutions and provide sufficient conditions for the permanency of the homoclinic  in the space of $C^1$ vector fields.  We also provide new sufficient conditions for the persistence  and multiple bifurcations of the  bounded solutions nearby. Our results can be  verified numerically and do not meet  the limitations of classic methods (like Melnikov`s integrals and Poincare`s map). 
\keywords{Homoclinic bifurcation \and Lyapunov-Schmidt reduction \and Saddle-node bifurcation\and Transcritical bifurcation}
 \subclass{34D10 \and 34C37}
\end{abstract}




\section{Introduction}
In the theory of differential equations, the homoclinics are usually known because of the structural unstability often imposing on systems. Some of the most studied chaotic motions arise from homoclinic bifurcations and many classic methods of the theory (like Melnikov`s integrals and Poincare`s maps, \cite{Wiggins}, \cite{Guckenheimer}) are developed for studying such behaviors. And today, homoclinics are still receiving attention from researchers of both the sciences and engineering areas (see for example \cite{Barrio,Chacon,Chen,Franca,Kowalczyk,Wu,CZhang,Zhang}). 
 
In the last decades, non-classic methods have been continuously developed by the researchers for studying the behavior of systems around a perturbed homoclinic. In the years after that, some studies have been concerned about providing sufficient conditions for the persistence of regular behaviors (like periodic orbits, bounded solutions and etc.) near a perturbed homoclinic. In an interesting study, \cite{Zhu}, Zhu and Zhang considered the $n$-dimensional differential equation $ \dot{x}=f(x)+g(x,t) $ with $f,g \in C^3$. They provided sufficient conditions for the existence of an invariant manifold which is spanned by homoclinic orbits. They called this invariant manifold as homoclinic finger-ring and showed that solutions on this manifold remain bounded.  
 
In this paper, we will provide a sufficient condition for the permanency of bounded solutions near a perturbed homoclinic and will explain how it can include the permanency of the homoclinic itself. We will also provide new sufficient conditions for the bifurcations of solutions nearby. Consider the system 
\begin{equation}\label{eq2} 
\dot{x}=f(x)+ \epsilon g(x,t),\hspace{1cm}x\in\mathbb{R}^2 
\end{equation} 
where $f=(f_1,f_2)$ and $g=(g_1,g_2)$ are respectively $C^2$ and $C^1$ w.r.t. $x$ and $g$ is bounded (but not necessarily periodic) w.r.t. $t$. We assume that the unperturbed system $\dot{x}=f(x)$ has a homoclinic orbit $\gamma(t)=\big(\gamma_1(t),\gamma_2(t)\big)$ based on the equilibrium $x_0=0$ with the corresponding eigenvalues $\pm\omega\neq 0 $. We will impose a decomposition on the corresponding variational equation and will develop results and methods to find bounded solutions and their bifurcations near $\gamma(t)$. 
Our method can be verified numerically and does not meet the limitations of classic methods (like Melnikov`s integrals and Poincare`s map). Apparently, this makes the results more appropriate for real-world applications in compare with former results. The main results are summarized below. We introduce the conditions: 
\begin{itemize}
 \item[\textbf{C1}: \ ]{ $\mathcal{F}_1=\int\limits_{\mathbb{R}}
\frac{1}{\Delta(s)}f\big(\gamma(s)\big) \wedge Df\big(\gamma(s)\big)\gamma(s)\,ds\neq0$.
 }
 \item[\textbf{C1$^{'}$}: \ ]{$\mathcal{F}_1^{'}=\int\limits_{\mathbb{R}}
\frac{1}{\Delta(s)}f\big(\gamma(s)\big)\wedge g\big(\gamma(s),s\big)\,ds\neq0$.
 \item[\textbf{C2}: \ ]{The map $\gamma \wedge f\big(\gamma\big)\in C^0_b(\mathbb{R},\mathbb{R}^2) $ is not identically zero. Here the wedge product of two vectors $ u=(a,b) $ and $ v=(a^{'}, b^{'}) $ is defined as $ u \wedge v =ab^{'} -ba^{'}$.}
  }
  \end{itemize}
 
 Note that, the assumption \textbf{C2} implies that $\gamma$ and $\gamma^{'}$ are linearly independent functions.
 \begin{thm}\label{th2} Consider system (\ref{eq2}).\\
 $(i)$  If \textbf{C1} holds  then for $0\!\leq\!|\epsilon|\!\!<\!\!<\!\!1$,  (\ref{eq2}) has a bounded solution $x(t)$ near $\gamma(t)$. Furthermore, $\|x(t)-\gamma(t)\|\rightarrow0$ as $\epsilon\rightarrow0$.\\
 $(ii)$ If \textbf{C1$^{'}$} and \textbf{C2}  hold then for $0\!\leq\!|\epsilon|\!\!<\!\!<\!\!1$, the system has a bounded solution $x(t)$ near $ \gamma(t)$. Furthermore, $\|x-\gamma\|\rightarrow0$ as $\epsilon\rightarrow0$.
\end{thm}
\begin{remark}
In part $(i)$ of theorem \ref{th2}, the condition \textbf{C1} is independent of $g$. This means that, if \textbf{C1} holds, then for any appropriate $C^1$ map $g(x,t)$ and $0\!\leq\!|\epsilon|\!\!<\!\!<\!\!1$, (\ref{eq2}) has a bounded solution near $\gamma(t)$. Especially, if $g=g(x)$ is independent of $t$, then this bounded solution is homoclinic. This means that the homoclinic of the unperturbed system in (\ref{eq2}) is permanence in the space of $C^1$ vector fields.
\end{remark}

Now  Let  $ \zeta(t)=\big(\zeta _{1}(t),\zeta _{2}(t)\big) $ be the unbounded  solution
of  $\dot{x}=Df\big(\gamma(t)\big)x$ and $\Delta(t)=\gamma_1^{'}(t)\zeta_2(t)-\gamma_2^{'}(t)\zeta_1(t)$. We introduce the following conditions applied in the next theorems.
\begin{itemize}
\item[\textbf{C3}: \ ]{$\mathcal{F}_3={\int\!\!\int}_{\mathbb{R}^2} \frac{f_2 \big(\gamma(s)\big)f_1 \big(\gamma(t)\big)}{\Delta(t)\Delta(s)}\Big[G(s,t)\wedge\hat{F}(s,t)\Big] \; ds dt\neq0 $
 \begin{equation*}
\text{Here}\hspace{0.3cm}\hat{F}(s,t)\!\!=\!\! \left(\!\! \begin{array}{c}
\left\langle  {\bigtriangledown f_1\big(\gamma(s)\big),\gamma(s)}  \right\rangle -f_1\big(\gamma(s)\big)\\
\\
\left\langle  \bigtriangledown f_2\big(\gamma(t)\big),\gamma(t)\right\rangle-f_2\big(\gamma(t)\big)
 \end{array}\!\! \right),\hspace{0.2cm}
 G(s,t) =
 \left(\! \begin{array}{c}
g_1\big(\gamma(s)\big) \\
\\
g_2\big(\gamma(t)\big)
 \end{array}\! \right).
 \end{equation*}
 }
\item[\textbf{C4}: \ ]{
\begin{eqnarray*}
\mathcal{F}_{4,1}&:=&\int _\mathbb{R} {\frac{1}{\Delta(s)}f_2\big(\gamma(s)\big)g_1\big(\gamma(s),s-\beta\big)} ds \!\neq\! 0,\\ \mathcal{F}_{4,2}&:=&{\int\!\!\!\int}_{\mathbb{R}^2} \frac{ f_1 \big(\gamma(t)\big)f_2 \big(\gamma(s)\big)}{\Delta(t)\Delta(s)}\Big[\bar{F}_2(t) \wedge \bar{F}_1(s) \Big]dtds\! \neq\!0,\\
\mathcal{F}_{4,3}&:=&\int _\mathbb{R} {\frac{f_2\big(\gamma(s)\big)}{\Delta(s)}\Big(\left\langle \bigtriangledown f_1\big(\gamma(s)\big),\gamma(s) \right\rangle - f_1(\gamma(s))\Big)ds}\neq 0.
\end{eqnarray*}
 }
 \item[\textbf{C4$^{'}$}: \ ]{
\begin{eqnarray*}
\mathcal{F}_{4,1}^{'}&:=&\int _\mathbb{R} {\frac{1}{\Delta(s)}f_1\big(\gamma(s)\big)g_2\big(\gamma(s),s-\beta\big)} ds \!\neq\! 0,\\ \mathcal{F}_{4,2}^{'}&:=&{\int\!\!\!\int}_{\mathbb{R}^2} \frac{ f_2 \big(\gamma(t)\big)f_1 \big(\gamma(s)\big)}{\Delta(t)\Delta(s)}\Big[\bar{F}_1(t) \wedge \bar{F}_2(s) \Big]dtds\! \neq\!0,\\
\mathcal{F}_{4,3}^{'}&:=&\int _\mathbb{R} {\frac{f_1\big(\gamma(s)\big)}{\Delta(s)}\Big(\left\langle \bigtriangledown f_2\big(\gamma(s)\big),\gamma(s) \right\rangle - f_2(\gamma(s))\Big)ds}\neq 0.
\end{eqnarray*}
 }
\item[\textbf{C5}: \ ]{
 $\mathcal{F}_{5}:= {\int\!\!\int}_{\mathbb{R}^2} \frac{ f_1 \big(\gamma(t)\big)f_2 \big(\gamma(s)\big)}{\Delta(t)\Delta(s)}\Big[\tilde{F}_2(t) \wedge \tilde{F}_1(s) \Big]dtds \neq0 $.
}
\end{itemize}
\[
\hspace{-1.2cm}\text{Here}\hspace{0.5cm}\bar{F}_k(s)=
 \left(\! \begin{array}{c}
 \sum_{i,j=1}^2 D^{2} _{x_i x_j} f_k \big(\gamma(s)\big)\\
\\
 \left\langle  \bigtriangledown f_k\big(\gamma(s)\big),\gamma(s)  \right\rangle-f_k\big(\gamma(s)\big)
 \end{array}\! \right),
\hspace{0.5cm}\hspace{0.5cm}k=1,2.
 \]
\[
 \tilde{F}_1(s)=
 \left(\! \begin{array}{c}
\bigtriangledown g_1 \big(\gamma(s),s-\beta\big)+\sum_{i,j=1}^2 D^{2} _{x_i x_j} f_1 \big(\gamma(s)\big)\Big(\kappa_{1}\gamma_j(s)+\kappa_{2j}+\kappa_{j3}\Big) \\
\\
\left\langle \bigtriangledown f_1(\gamma(s)),\gamma(s)\big) \right\rangle -f_1(\gamma(s))
 \end{array}\!\right),
\]
 \[
 \tilde{F}_2(t)= \left(\! \begin{array}{c}
\bigtriangledown g_2 \big(\gamma(s),s-\beta\big)+\sum_{i,j=1}^2 D^{2} _{x_i x_j} f_2 \big(\gamma(t)\big)\Big(\kappa_{1}\gamma_j(t)+\kappa_{2j}+\kappa_{3j}\Big)\\
\\
\left\langle\bigtriangledown f_2(\gamma(t)),\gamma(t)\big)\right\rangle -f_2(\gamma(t))
 \end{array}\!\right),
 \]
 \begin{eqnarray*}
\kappa_{1}\! &=&\!  -\frac{\int _\mathbb{R} {f_2\big(\gamma (s)\big) g_1\big(\gamma(s),t-\beta\big)/\Delta(s)ds}}{\int _\mathbb{R} {f_2\big(\gamma (s)\big) \Big(\left\langle \bigtriangledown f_1(\gamma(s)),\gamma(s)\right\rangle- f_1(\gamma(s))\Big)/\Delta(s)ds}},\\
 \kappa_{2j}\! & =&\! \zeta _j (t)\int_{-\infty}^{t}\!\! \frac{1}{\Delta(s)}f\big(\gamma(s)\big)\wedge Df\big(\gamma(s)\big) \gamma(s) ds\\
&+& f_j\big(\gamma(t)\big)\int\limits_0^t  \frac{1}{\Delta(s)}\big[Df\big(\gamma(s)\big)\gamma(s)-f\big(\gamma(s)\big)\big] \wedge \zeta(s)ds,\hspace{1cm}(j=1,2)\\
\kappa_{3j}&=&\zeta _j (t)\int_{-\infty}^t\frac{1}{{\Delta (s)}}f\big(\gamma(s)\big)\wedge g(\gamma(s),s-\beta) ds\\
&+& f_j\big(\gamma(t)\big)\int\limits_0^{t} \frac{1}{\Delta(s)}g(\gamma(s),s-\beta) \wedge \zeta(s) ds,\hspace{1cm}(j=1,2)
\end{eqnarray*}
\begin{itemize}
\item[\textbf{C6}: \ ]{ $g(x,t)$ is not $t$-constant in a neighborhood of $\gamma(t)$, that is , there exists an open subset  $U\subseteq\mathbb{R}^2$ containing $\gamma(t)$ such that for $(x,t_1),(x,t_2)\in U\times\mathbb{R}$, $t_1\neq t_2$ implies $g(x,t_1)\neq g(x,t_2)$.} 
\end{itemize}

\begin{thm}\label{th3}
 If the conditions \textbf{C3}, \  (\textbf{C4} or \textbf{C4$^{'}$}), \textbf{C5} and \textbf{C6} hold then  depending on the sign of $0\!\leq\!|\epsilon|\!\!<\!\!<\!\!1$,  (\ref{eq2})  has  two bounded solutions $x_1(t)$ and $x_2(t)$ near $ \gamma(t)$. Furthermore, $\|x_{1,2}-\gamma\|\rightarrow0$ as $\epsilon\rightarrow0$.
\end{thm}
\begin{thm}\label{thm4}
If the condition \textbf{C3} fails and the conditions (\textbf{C4} or
\textbf{C$^{'}$4}), \textbf{C5}  and \textbf{C6} hold then independent of the  sign of  $0\!\leq\!|\epsilon|\!\!<\!\!<\!\!1$,  (\ref{eq2})  has two bounded solutions $ x_1(t) $ and $ x_2(t) $ near $ \gamma(t) $ such that $\|x_{1,2}-\gamma\|\rightarrow0$ as $\epsilon\rightarrow0$.
\end{thm}
\section{Lyapunov-Schmidt reduction}
Here we explain the  Lyapunov-Schmidt reduction method in brief.  Consider the sufficiently differentiable map
\[
\varphi:\mathbb{R}^n \times \mathbb{R} \longrightarrow \mathbb{R}^n
,\hspace{1cm}\varphi:(x,\epsilon) \longmapsto \varphi (x,\epsilon).
\]
We denote the first order derivative of $\varphi$ w.r.t.  $ x $  by $ D_x \varphi $. Assume that the unperturbed map $ x \mapsto \varphi (x,0) $ has a zero $ x=x_0 $, i.e. $\varphi(x_0,0)=0$. The question is whether the perturbed map has a zero too? One of the main tools for answering this question is the implicit function theorem; but,  if the conditions of the implicit
function theorem do not hold then the Lyapunov-Schmidt reduction is an effective tool.

Let $N(L)$ be the kernel of $ L:=D_x\varphi (x_0,0)$ with  $dimN(L)=k>0 $ and a complementary subspace $N^\perp (L) $.
Also suppose that $ p:\mathbb{R}^n  \rightarrow R(L) $ is a projection with complement $ I-p:\mathbb{R}^n  \rightarrow R^\perp (L) $. Here $ R(L) $
and $ R^\perp (L) $ are respectively the range of $ L $ and a complementary subspace of $R(L)$.
Since $ \mathbb{R}^n=N(L)\oplus N^\perp (L) $ so we can decompose $x\in\mathbb{R}^n$  to its components $\xi\in N(L)\approx\mathbb{R}^k$ and $\eta\in N^\perp(L)\approx\mathbb{R}^{n-k}$, hence we have $x=\xi+\eta$.  Now consider the map
\[
\upsilon:\mathbb{R}^k \times \mathbb{R}^{n-k} \times \mathbb{R} \rightarrow R(L),\hspace{0.5cm}
\upsilon:(\xi ,\eta,\epsilon)\mapsto p \varphi (\xi +\eta+x_0,\epsilon).
\]
Since $\upsilon(0,0,0)=0$ and $pL=D_\eta \upsilon(0,0,0) :  N^\perp (L)  \rightarrow R(L) $ is an isomorphism, so the implicit function theorem implies that there are sufficiently small neighborhoods $A$ of $ (\xi,\epsilon)=(0,0)$,  $B$ of $\eta=0$ and a $C^2$  map
$\eta: A\rightarrow B$
such that $\eta(0,0)=0$ and $ \eta=\eta(\xi, \epsilon) $ is the unique solution of
\[\upsilon(\xi,\eta,\epsilon)=p \varphi (\xi + \eta+x_0,\epsilon)=0.\]
Therefore, the Lyapunov-Schmidt reduction method reduces the problem of finding zeros for $ \varphi $ to finding zeros of the function
\[
\tau:\mathbb{R}^k \times \mathbb{R}  \rightarrow R^\perp (L),\hspace{0.5cm}
\tau:(\xi,\epsilon)\mapsto(I-p)\varphi(\xi+\eta(\xi, \epsilon)+x_0, \epsilon).
\]

Note that $ \tau(0,0)=0 $;
hence, if there exists a function $ \xi (\epsilon) $ such that $\xi(0)=0$ and $ \tau(\xi(\epsilon),\epsilon)=0 $ then $ x(\epsilon)= \xi(\epsilon)+\eta\big(\xi(\epsilon),\epsilon\big)+x_0$ is a solution for $\varphi(x,\epsilon)=0$.
The function  $ \tau $ is called a bifurcation function for $\varphi(x,\epsilon)=0$. Note that $x=x(\epsilon) $ could not be obtained directly by  the implicit function theorem, because $L:=D_x\varphi (x_0,0)$ has a nontrivial kernel  (see \cite{Chicone}  for more details and \cite{Hale} for development of the theory over Banach spaces).

\section{Variational equation}\label{1}
In this section,  we consider the variational equation of (\ref{eq2}) along $\gamma(t)$, i.e. $\dot{z}=A(t)z$ with $A(t)=Df\big(\gamma(t)\big)$. We will look for bounded solutions of the equation
\begin{equation}\label{eq5}
\dot{z}-A(t)z=F(t), \hspace{1cm} F=(F_1,F_2) \in C_b^0 (\mathbb{R},\mathbb{R}^2).
\end{equation}
Here $ C_b^{0} (\mathbb{R},\mathbb{R}^2)$ shows the Banach space of bounded continuous functions from $\mathbb{R}$ to $\mathbb{R}^{2} $ with
$ \left\| {F} \right\| =\mathop {\sup }\limits_{t \in \mathbb{R}} \left| {F(t)} \right| $.
\begin{lem}\label{lem1}
Let $X(t)$ be a Fundamental matrix of the variational equation $\dot{z}=A(t)z$ and $\Delta(t)=\det X(t)$. Then $\Delta(t)$ is bounded with $\Delta(+\infty)=\Delta(-\infty)$.
\end{lem}
\begin{proof}
It is enough to show that $\dot{z}=A(t)z$ has a fundamental matrix with bounded determinate and $\Delta(+\infty)=\Delta(-\infty)$. From \cite[lemma 1]{Gruendler1995}, there exist a fundamental matrix $X(t)$  and an invertible constant matrix $C$ such that
\begin{equation}\label{omid1}\lim_{t\rightarrow+\infty}X(t)\left(\begin{array}{lr}
e^{\omega t}&0\\
0&e^{-\omega t}\end{array}\right)=C.\end{equation}
We may complete the steps of the proof of the lemma to verify that
\begin{equation}\label{omid2}\lim_{t\rightarrow-\infty}X(t)\left(\begin{array}{lr}
e^{-\omega t}&0\\
0&e^{\omega t}\end{array}\right)=C.\end{equation}
Thus we have
\[\Delta(+\infty)=
\lim_{t\rightarrow+\infty}\det X(t)=\det C =\lim_{t\rightarrow-\infty}\det X(t)=\Delta(-\infty).\]
 \qed
\end{proof}
Since  $ \gamma^{'}(t) $ is a solution of  $\dot{z}=A(t)z$ so we can find a fundamental matrix $X(t)$  as below
\begin{equation}\label{eq6}
X(t)=\left(
\begin{array}{lr}
 \gamma_{1}^{'}(t) \ \ & \ \  \zeta _{1}(t) \\
 \gamma_ {2}^{'}(t)  \ \ & \ \ \zeta _{2}(t) \\
\end{array}\right)
\end{equation}
with  $ X(0)=Id $. Here $ \zeta(t)=\big(\zeta _{1}(t),\zeta _{2}(t)\big) $ is any solution independent of $\gamma^{'}(t)$; hence both $\zeta_1(t)$ and $\zeta_2(t)$ are unbounded as $t\rightarrow\pm\infty$.
Thus,  from (\ref{omid1}) and (\ref{omid2}), when $t\rightarrow\pm\infty$, we have
\begin{equation}
\label{o1new}
\gamma^{'}(t)e^{+ \omega |t| }  \to constant, \hspace{1cm}
\zeta(t)e^{- \omega |t| }  \to constant.
\end{equation}
Furthermore, there exists $ k>0 $ such that
\begin{eqnarray}
\label{eq6.11}
&\text{for}& \ \  t\geq s \geq 0, \hspace{0.5cm} |\gamma _i^{'} (t)\zeta _j (s)| < k \exp\big(- \omega (t - s)\big),  \hspace{0.4cm} i,j = 1,2, \\
\label{eq6.12}&\text{for}& \ \  t\leq s \leq 0, \hspace{0.5cm} |\gamma _i^{'} (t)\zeta _j (s)| < k \exp\big(\omega (t - s)\big),  \hspace{0.8cm}  i,j = 1,2.
\end{eqnarray}

\begin {lem}\label{the1}
The system (\ref{eq5}) has a bounded solution Iff
\begin{equation}\label{eq17}
\int  \limits_ {\mathbb{R}}
\frac{1}{{\Delta (s)}} f\big(\gamma(s)\big)\wedge F(s)ds=0.
\end{equation}
\end{lem}
\begin{proof}
: The proof is under construction and will be appear after publishing.
\end{proof}
A direct conclusion from the above lemma is that  bounded solutions of (\ref{eq5}) is
obtained by 
\begin{eqnarray}
\label{o3new}
z_1(t)\!&=&\!\gamma_1^{'}(t)\Big(x_2\!-\!\!\int_0^t\!\frac{\zeta(s)}{\Delta(s)}\wedge F(s)ds\Big)
\!\!+\!\zeta_1(t)\!\!\int_{-\infty}^t\!\frac{\gamma^{'}(s)}{\Delta(s)}\wedge F(s)ds.\\
z_2(t)\!&=&\!\gamma_2^{'}(t)\Big(x_2\!-\!\!\int_0^t\!\frac{\zeta(s)}{\Delta(s)}\wedge F(s)ds\Big)\!\!+\!\zeta_2(t)\!\!\int_{-\infty}^t\!\frac{\gamma^{'}(s)}{\Delta(s)}\wedge F(s)ds.
\label{o4new}
\end{eqnarray}

\section{Bifurcation map and homoclinic bifurcations}
In this section, we consider the equation (\ref{eq2}) and  look for its bounded solutions near $\gamma(t)$. We will find a bifurcation map for the system and will investigate its homoclinic bifurcations. To this end, let $x(t)$ be a solution of (\ref{eq2}) of the form
\begin{equation}\label{eq18bb}
x(t-\beta)=\alpha\gamma(t)+z(t),
\end{equation}
with $\alpha \approx 1$ and   $0\leq|\beta|<\!\!<1$. By replacing $x(t-\beta)$ in (\ref{eq2}) we get
\begin{equation}\label{eq11n}
\dot{z} - A(t)z = F(t,z,\alpha ,\beta ,\epsilon ),
\end{equation}
where $A(t)=Df\big(\gamma(t)\big)$ and
\begin{eqnarray*}
F(t,z,\alpha ,\beta ,\epsilon )&=& f\big(\alpha \gamma(t)\big)+Df\big(\alpha \gamma(t)\big)z-\alpha f\big(\gamma(t)\big)-A(t)z\\
&+ &\epsilon g\big(\alpha \gamma(t)+z,t-\beta\big)+O(z^2,\epsilon z).
\end{eqnarray*}
It is easy to see that $x(t)$ is bounded Iff $z(t)$ is bounded; furthermore, $\|x-\gamma\|=\|z\|$.
Thus, if  $z(t)$ is a nontrivial bounded solution of (\ref{eq11n}) near zero then $x(t)$ is a nontrivial bounded solution of (\ref{eq2}) near $\gamma$.  Because of the exponential dichotomy of the variational equation $\dot{z}=A(t)z$,  the existence of a nontrivial bounded solution for (\ref{eq11n}) which is enough near to zero, equals to the existence of a homoclinic bifurcation for (\ref{eq2}).

On the other hand, from lemma \ref{the1}, the solution $z(t)$ of (\ref{eq11n}) is  bounded  Iff
\[\int_\mathbb{R}\frac{f\big(\gamma(t)\big)}{\Delta(t)}\wedge F(t,z(t),\alpha ,\beta ,\epsilon )dt=0.\]
Thus, in order to have $x(t)$ bounded and near $\gamma(t)$, we must investigate solutions $z(t)$ of (\ref{eq11n}) near zero such that the above integral equality holds. For this purpose, we define the linear projection $p$ as below
\begin{equation}\label{eq14}
p:C_b^0(\mathbb{R}, \mathbb{R}^2) \to C_b^0(\mathbb{R}, \mathbb{R}^2),\hspace{0.5cm}
p:F(t)  \mapsto \frac{\Delta^2(t)}{\|\gamma^{'}\|_2^2}J(t) \int\limits_{\mathbb{R}} {J(s).F(s)ds},
\end{equation}
where
\[ J(s) = \frac{1}{\Delta(s)}\left( \begin{array}{c}
- {{\gamma _2^{'} (s)}} \\
 {{\gamma _1^{'} (s)}} \\
 \end{array}\right).\]
We can also consider (\ref{eq5}) as the operator equation $(Lz)(t)= F(t)$ where
\begin{equation}
\label{eq9}
L:C_b^1(\mathbb{R},\mathbb{R}^2)\rightarrow C_b^0(\mathbb{R},\mathbb{R}^2),
\hspace{1cm}
(Lz)(t)=\dot{z}(t)-A(t)z(t).
\end{equation}
The  lemma \ref{the1} implies that the enough and sufficient condition for a map  $F(t)=\big(F_1(t),F_2(t)\big)\in C^0_b(\mathbb{R},\mathbb{R}^2)$ belongs to $R(L)$ is that $pF(t)=0$, i.e. $R(L)=N(p)$.
Also it is easy to see that
\[N(L)=\{\xi\gamma^{'}(t) \ : \ \xi\in\mathbb{R}\}\hspace{0.3cm}\text{and}\hspace{0.3cm}R(p)=\{\xi\Delta(t)\left(\begin{array}{c}-\gamma_2^{'}(t)\\
\gamma_1^{'}(t)\end{array}\right) \ : \ \xi\in\mathbb{R}\}\]
are one dimensional subspaces, so we can consider $\xi\in\mathbb{R}$ as an element of
$N(L)$ (or $R(p)$). Since $I-p:C_b ^0 (\mathbb{R},\mathbb{R}^2)= N(p)\oplus R^\perp(L)\rightarrow N(p)$ so $ k$, the inverse of $L:N(L)^\perp\rightarrow N(p)$, is a well defined linear isomorphism as below
\[k:(I-p)C_b^0 (\mathbb{R}, \mathbb{R}^2) \rightarrow  N(L)^\perp.\]
This enables us to decompose $z\in C^1_b(\mathbb{R},\mathbb{R}^2)$ as $ z=\xi +\eta \in N(L) \oplus N^\perp (L)$; then by using the Lyapunov-Schmidt reduction, the problem of finding  bounded solutions for (\ref{eq11n}) is equivalent to the solving  of the system of equations
\begin{eqnarray}
\label{eq15}0&=&pF(t,\xi +\eta ,\alpha ,\beta ,\epsilon ),\\
\label{eq16}\eta(t)&=& k(I-p)F(t,\xi +\eta ,\alpha ,\beta ,\epsilon ).
\end{eqnarray}

Let
\[
\begin{array}{l}
G:C^0_b(\mathbb{R},\mathbb{R}^2)\times\mathbb{R}^3\rightarrow N^\perp(L),\\
G(\xi +\eta,\alpha ,\beta ,\epsilon ):=\eta-k(I-p)F(\xi+\eta,\alpha ,\beta ,\epsilon ).\end{array}\]
Then it is easy to check that $ G(0,1,0,0)=0 $ and $ D_{\eta}G(0,1,0,0)=I $. Hence, by applying the implicit function theorem,
there exist a neighborhood $U$ of $(0,1,0,0)\in N(L)\times\mathbb{R}^3$ and a unique map $ \eta:U\rightarrow N^\perp(L)$ such that $ \eta(0,1,0,0)=0$ and $\eta(\xi,\alpha ,\beta ,\epsilon )$ is the unique solution of (\ref{eq16})  in $U$.
 Replacing $\eta(\xi,\alpha ,\beta ,\epsilon )$ in (\ref{eq15}), finally we obtain the bifurcation function as below
 \[ pF \big(t,\xi+\eta(\xi,\alpha ,\beta ,\epsilon ),\alpha ,\beta ,\epsilon \big)=0 \]
 or equivalently
\begin{equation}\label{eq15a}
B(\xi,\alpha ,\beta ,\epsilon )=\int_\mathbb{R}J(s)F\big(s,\xi+\eta(\xi,\alpha ,\beta ,\epsilon ),\alpha ,\beta ,\epsilon \big)ds =0.
\end{equation}
It is useful to note that, by differentiating from \eqref{eq16} w.r.t. $\xi$ we find that 
\[D_\xi\eta=[I-k(I-p)\frac{\partial F}{\partial z}]^{-1}k(I-p)\frac{\partial F}{\partial z}\]
which implies that
\begin{equation}\label{eqetan}
D_\xi\eta(0,1,0,0)=0.
\end{equation}
\begin{proof}[\textbf{Proof of theorem \ref{th2}}]
The proof is under construction and will be appear after publishing.
\end{proof}
\begin{proof}[\textbf{Proof of theorem \ref{th3}}]
The proof is under construction and will be appear after publishing.
\end{proof}
\begin{proof}[\textbf{Proof of theorem \ref{thm4}}]
The proof is under construction and will be appear after publishing.
\end{proof}

We end this section by giving an application of the theorem \ref{th2} to the power-law nonlinear oscillatory system (\ref{eq19a}). Bifurcations of such  systems have been widely studied in engineering and sciences (see, for example, \cite{Chen,Siewe,Wang,Li,Kuznetsov,Zhou}). They mostly concerned on investigating of chaotic motions by using the Melnikov method, however, our purpose here is finding a bounded solution near the perturbed homoclinic. Consider the system
\begin{equation}\label{eq19a}
\dot{x}=y,\hspace{1cm}
\dot{y}=\nu x -\mu x^{p+1}+\epsilon g(x,y,t) 
\end{equation}
where $\nu, \ \mu$ are positive parameters, $p>1$ is the integral power of the strongly nonlinear term and $ g(x,y,t)$ is a self-excited force and damping. Let $\gamma(t)=\big(x(t),y(t)\big)$ be the homoclinic of the unperturbed system, ($\epsilon=0$), with $\gamma(0)=(x_{\max},0)$, $x_{\max}=\sqrt[p]{(p+2)\nu/2\mu}$. The orbit of $\gamma(t)$ is the graph of the functions
\[y_\pm(x)=\pm x\sqrt{\nu-\frac{2\mu}{p+2}x^p},\hspace{1cm} 0\leq x\leq x_{\max}.\]
It must be noted that, since (\ref{eq19a}) is Hamiltonian so the determinant $\Delta(t)$ is constant; thus it can be omitted  from the computations through the conditions.
It is easy to see that \textbf{C2} holds, thus we calculate $\mathcal{F}^{'}_1$ in \textbf{C1$^{'}$} for (\ref{eq19a}) and obtain:
\begin{eqnarray*}
\mathcal{F}_1^{'}&=&\int_{\gamma(t)}g(x,y,t)dx\\
&=&\int_0^{x_{\max}}g_2\big(x,y_+(x),t(x)\big)-g_2\big(x,y_-(x),-t(x)\big)dx
\end{eqnarray*}
where $t(x)\leq 0$ shows the time  $x(t)=x$. Thus,  if one of the following assumptions holds then $\mathcal{F}_1^{'}\neq0$ and theorem \ref{th2} implies that (\ref{eq19a}) has a bounded solution near  $\gamma$, for $|\epsilon|<\!\!<1$.
\begin{quote}
\textbf{A1:} The function $g(x,y,t)$ is increasing w.r.t. $y$ and it is even and bounded w.r.t. $t$.\\
\textbf{A2:} The function $g(x,y,t)\geq0$ is even w.r.t. $y$ and it is odd and bounded w.r.t. $t$.\\
\textbf{A3:} The function $g(x,y,t)\geq0$ is odd w.r.t. $y$ and it is even and bounded w.r.t. $t$.
\end{quote}

\section{Conclusion}
Although homoclinic bifurcations are mostly known because of the chaotic behavior they might impose on a system, here we studied them from the bifurcation theory point of view. The chaotic behavior, of course, exists if the phase space can be decomposed to the direct sum of the stable and unstable subspaces of the corresponding variational equation. We left this for further studies in another paper.   At this point, our results guarantee the existence and bifurcations of bounded solutions near an unperturbed homoclinic.

In a special case, when the function $g(x,t)$ in (\ref{eq2}) is $T$-periodic w.r.t. $t$, the bounded solution implied by theorem \ref{th2} is a homoclinic orbit based on the unique hyperbolic $T$-periodic solution near the origin. The same statement is valid for bounded solutions in theorems \ref{th3} and \ref{thm4}. In a more general case, If (\ref{eq2}) is time-dependent then from (\ref{eq18b}) and the structure of the proof of theorems \ref{th3} and \ref{thm4}, the bounded solutions $x_i(t)$, $i=1,2$, are not a time rescale of each other. Thus these solutions are distinct.

If (\ref{eq2}) is autonomous (i.e. $g(x,t)=\bar{g}(x)$ is independent of $t$) then  the bounded solution implied by the theorem \ref{th2} is a unique homoclinic orbit based on a unique hyperbolic fixed point near the origin. In this case the two distinct solutions of  theorem \ref{th3} and \ref{thm4} are time rescales of each other.

Among the conditions of this paper,  $\textbf{C5}$ is probably the hardest to verify. It needs usually to be verified numerically due to the unbounded solution $\zeta(t)$ which appears in the formula. What is interesting is that, although the conditions \textbf{C1, C4} and \textbf{C4$^{'}$} are generic, they always fail for Hamiltonian systems because of  $\mathcal{F}_{4,2}=\mathcal{F}_{4,2}^{'}=0$.
 Thus, for applying the theorems  \ref{th3} and \ref{thm4} for a Hamiltonian system, we have to  make a change of variables and bring the system into a non-Hamiltonian system. Indeed the theorems of this paper are easier to apply for non-Hamiltonian systems.

Finally, It is useful to note that, although the conditions of this paper are formulated by terms of usual integrals on $\mathbb{R}$ or $\mathbb{R}^2$, for a real application, it is more appropriate to consider them as  integrals on the curve $\gamma(t)$ (see the example of section 4). The later mentioned forms are considered because they are more suitable for the proofs of  theorems.

\section{Conflict of interest}
The authors declare that they have no conflict of interest.


\begin{thebibliography}{1}
\bibitem{Barrio}
Barrio,~R., Ibanez,~S., Perez,~L., and Serrano,~S.,
\newblock Spike-addingstructurein fold/hom bursters.
\newblock {\em Commun Nonlinear Sci Numer Simulat}, 2020, vol. 83, pp. 1--15.

\bibitem{Chacon}
Chacon,~R., Miralles,~J.J., Martinez,~J.J., and Balibrea,~ F.,
\newblock Tamingchaosin dampeddrivensystemsbyincommensurate excitations.
\newblock {\em Commun Nonlinear Sci Numer Simulat}, 2019, vol. 73, pp. 307--318.


\bibitem{Chen}
Chen,~Y.Y., Chen,~S.H., and Zhao,~W.,
\newblock Constructing explicit homoclinic solution of oscillators: An
  improvement for perturbation procedure based on nonlinear time
  transformations.
\newblock {\em Commun Nonlinear Sci Numer Simulat}, 2017, vol. 48, pp. 23--139.

\bibitem{Chicone}
Chicone,~C.,
\newblock Lyapunov-schmidt reduction and melnikov integrals for bifurcation of
  periodic solutions in coupled oscillators.
\newblock {\em Publication of university of Missouri}, 2004, pp. 1--34.

\bibitem{Chow}
Chow,~S.N., Hale,~J.K., and Mallet-Parret,~J.,
\newblock An example of bifurcation to homoclinic orbits.
\newblock {\em J.~Differential Equations}, 1980, vol. 37, pp. 351--373.

\bibitem{Franca}
Franca,~M., and Pospisil,~M.,
\newblock New global bifurcation diagrams for piecewise smooth systems:
  Transversality of homoclinic points does not imply chaos.
\newblock {\em J.~Differential Equations}, 2019, vol. 266, pp. 1429--1461.

\bibitem{Gruendler1995}
Gruendler,~J.,
\newblock Homoclinic solutions for autonomous ordinary differential equations
  with nonautonomous perturbations.
\newblock {\em J. Differential Equations}, 1995, vol. 122, pp. 1--26.

\bibitem{Guangping}
Guangping,~L., Juan,~L., and Changrong,~Z.,
\newblock The transversal homoclinic solutions and chaos for stochastic
  ordinary differential equations.
\newblock {\em J.~Mathematical Analysis and Applications}, 2014, vol. 412, pp. 301--325.

\bibitem{Guckenheimer}
Guckenheimer,~J., and Holmes,~P.,
\newblock Nonlinear oscillations, Dynamical systems, and Bifurcations of vector fields
\newblock {\em Springer Verlag New York}, 1983.

\bibitem{Hale}
Hale,~J.K., and Spezamiglio,~A.,
\newblock Perturbation of homoclinic and subharmonics in duffing's equation.
\newblock {\em Nonlinear Analysis, Theory, Methods and Applicarion}, 1985, vol. 9, no. 2, pp. 181--192, .

\bibitem{Kowalczyk}
Kowalczyk,~P.,
\newblock The dynamics and event-collision bifurcations in switched control systems with delayed switching.
\newblock {\em Physica D: Nonlinear Phenomena}, 2020, vol. 406, pp. 1--11.

\bibitem{Kuznetsov}
Kuznetsov,~A.P., and Roman,~J.P.,
\newblock Properties of synchronization in the systems of non-identical coupled
  van der pol and van der pol-duffing oscillators broadband synchronization.
\newblock {\em Physica D}, 2009, vol. 238, pp. 1499--1506.

\bibitem{Li}
Li,~J., Xu,~W., Yang,~X., and Sun,~Z.,
\newblock Chaotic motion of van der pol-mathieu-duffing system under bounded
  noise parametric excitation.
\newblock {\em J.~Sound and Vibration}, 2008, vol. 309, pp. 330--337.

\bibitem{Lin}
Lin,~X.B., Long,~B., and Zhu,~C.,
\newblock Multiple transverse homoclinic solutions near a degenerate homoclinic
  orbit.
\newblock {\em J.~Differential Equations}, 2015, vol. 259, pp. 1--24.

\bibitem{Siewe}
Siewe Siewe,~M., Moukam Kakmeni,~M., Tchawoua,~C., and Woafo,~P.,
\newblock Bifurcations and chaos in the triple-well $ \phi^6 $-van der pol
  oscillator driven by external and parametric excitations.
\newblock {\em Physica A}, 2005, vol. 357, pp. 383--396.

\bibitem{Wang}
Wang,~R., Deng,~J., and Jing,~Z.,
\newblock Chaos control in duffing system.
\newblock {\em Chaos, Solitons and Fractals}, 2006, vol. 27, pp. 249--257.

\bibitem{Weinian}
Weinian,~Z.,
\newblock Bifurcation of homoclinics in a nonlinear oscillation.
\newblock {\em Acta Mathematica Sinica}, 1989, vol. 5, pp. 170--184.

\bibitem{Wiggins}
Wiggins,~S.,
\newblock Introduction to applied nonlinear dynamical systems and chaos.
\newblock {\em USA: Springer-verlage, New York}, 2003, 3th ed..

\bibitem{Wu}
Wu,~T., and Yang,~X.S.,
\newblock On the existence of homoclinic orbits in n-dimensional piecewise
  affine systems.
\newblock {\em Nonlinear Analysis: Hybrid Systems}, 2018, vol. 27, pp. 366--389.

\bibitem{CZhang}
Zhang,,~C., Harne,~R.L., Li,~B.,  and Wang,~K.W.,
\newblock Statistical  quantification  of  DC  power  generated  by  bistable  piezoelectric  energy harvesters when driven by random excitations.
\newblock {\em J. of Sound and Vibration}, 2019, vol. 442, pp. 770--786.

\bibitem{Zhang}
Zhang,~Q., and Liu,~C.,
\newblock Homoclinic orbits for a class of first order nonperiodic hamiltonian
  systems.
\newblock {\em Nonlinear Analysis: Real World Applications}, 2018, vol. 41, pp. 34--52.

\bibitem{Zhou}
Zhou,~L., and Chen,~F.,
\newblock Chaotic motions of the duffing-van der pol oscillator with external
  and parametric excitations.
\newblock {\em Hindawi Publishing Corporation Shock and Vibration}, 2014, vol. 2014, pp. 1--5.

\bibitem{Zhu}
Zhu,~C., and Zhang,~W.
\newblock Homoclinic finger-rings in $ \mathbb{R}^n $.
\newblock {\em J.~Differential Equations}, 2017, vol. 263, pp. 3460--3490.
\end{thebibliography}
\end{document}